\documentclass[12pt,reqno]{amsart}
\usepackage{amsmath,amssymb,amsthm,amsfonts,enumerate}
\usepackage{bm}
\usepackage{mathrsfs}
\usepackage{xcolor}
\usepackage[colorlinks=true, linkcolor={blue},citecolor={purple},urlcolor={teal}]{hyperref}
\usepackage{cite}
\usepackage[left=1.5in,right=1.5in,top=1.in,bottom=1.in]{geometry}

\newcommand{\N}{\mathbb{N}}

\newtheorem{thm}{Theorem}[section]
\newtheorem{lem}[thm]{Lemma}

\theoremstyle{definition}
\theoremstyle{remark}

\newtheorem{defn}{Definition}[section]

\newtheorem{rem}[defn]{Remark}

\numberwithin{equation}{section}



\begin{document}

\title[Convergence exponents of the digits in Gauss like system]{Multifractal analysis of the convergence exponents for the digits in $d$-decaying Gauss like dynamical systems}

\author {Kunkun Song}
\address{Key Laboratory of Computing and Stochastic Mathematics (Ministry of Education),
School of Mathematics and Statistics,
Hunan Normal University, Changsha, 410081, P.R. China}
\email{songkunkun@hunnu.edu.cn}

\author{Mengjie Zhang$^{*}$}
\address{School of Mathematics and Statistics, Henan University of Science and Technology, 471023 Luoyang, P.R. China}
\email{mengjie\_zhang@haust.edu.cn}

\thanks {* Corresponding author}

\begin{abstract}
Let $\{a_n(x)\}_{n\geq1}$ be the sequence of digits of $x\in(0,1)$ in infinite iterated function systems with polynomial decay of the derivative. We first study the multifractal spectrum of the convergence exponent defined by the sequence of the digits $\{a_n(x)\}_{n\geq1}$ and the weighted products of distinct digits with finite numbers respectively, and then calculate the Hausdorff dimensions of the intersection of sets defined by the convergence exponent of the weighted product of distinct digits with finite numbers and sets of points whose digits are non-decreasing in such iterated function systems.
\end{abstract}

\maketitle

\section{Introduction}

The study of infinite iterated function systems (iIFS) on the unit interval including the classical Gauss infinite iterated function system, has always been a fundamental and important subject in ergodic theory and number theory. In the past several decades, there are many works in diverse directions concerning the dimension and measure theory in infinite iterated function systems because of the establishments of Mauldin and Urba\'{n}ski, who investigated the infinite conformal iterated function systems in \cite{Mau95,MU03,MU96}. In this note, we shall consider certain sets from the viewpoint of multifractal analysis in infinite iterated function systems with polynomial decay of the derivative, namely the \emph{$d$-decaying Gauss like iterated function systems}. Below we introduce the definitions and relevant results of such systems, then state our main results.

\subsection{$d$-decaying Gauss like iterated function systems}
Let $\{f_n\}_{n\geq1}$ be a sequence of $C^{1}$ functions with $f_n:[0,1]\rightarrow[0,1]$ satisfying the following.
\begin{itemize}
\item[(\romannumeral1)]\emph{Open set condition}: for any $i\neq j\in\N$,\ $f_i((0, 1))\cap f_j((0, 1))=\emptyset$;

\medskip

\item[(\romannumeral2)]\emph{Contraction property}: there exists $m\in \N$ and a real number $0<\rho<1$ such that for all $(a_1, . . . , a_m) \in \N^m$ and $x\in[0,1]$,
\[
0<|(f_{a_1}\circ\cdots\circ f_{a_m})'(x)|\leq\rho<1;
\]

\medskip

\item[(\romannumeral3)]\emph{Regular property}: there exists $d>1$ such that for any $\varepsilon>0$, we can find constants $C_1=C_1(\varepsilon),\ C_2=C_2(\varepsilon)>0$ such that for $i\in\N$ there exist constants $\xi_i,\lambda_i$ such that for all $x\in[0,1]$, we have
\[
\xi_i \leq |f_i'(x)| \leq \lambda_i\quad \text{and} \quad
\frac {C_1} {i^{d+\varepsilon}} \leq \xi_i \leq \lambda_i \leq \frac {C_2} {i^{d-\varepsilon}}.
\]
\end{itemize}
We call such system a \emph{$d$-decaying iIFS} as defined by Jordan and Rams \cite{JR}. Moreover, it is called \emph{Gauss-like} if the system in addition fulfils
\begin{itemize}
\item[(\romannumeral4)] $\bigcup_{i=1}^\infty f_i([0,1]) = [0,1]$, \text{and}\ $f_{i}(x)<f_{j}(x)$ whenever $i>j$;

\medskip

\item[(\romannumeral5)]\emph{Bounded distortion property (BDP)}: there exists a constant $\kappa\geq1$ such that for every $n\in\mathbb{N}$ and $(a_1, . . . , a_n)\in \N^n$ we have
\[
|(f_{a_1}\circ\cdots\circ f_{a_n})'(x)|\leq\kappa|(f_{a_1}\circ\cdots\circ f_{a_n})'(y)|
\]
for all $x,y\in[0,1]$.
\end{itemize}

\medskip

With the above definition, clearly, there is a natural projection $\Pi:\ \mathbb{N}^{\mathbb{N}}\rightarrow[0,1]$ defined as
\[
\Pi(\underline{a})=\lim\limits_{n\rightarrow\infty}f_{a_{1}}\circ\cdots\circ f_{a_{n}}(1)
\]
for all $\underline{a}=\{a_n\}_{n\geq1}\in\mathbb{N}^{\mathbb{N}}$. Then each $x\in\Pi(\mathbb{N}^{\mathbb{N}})$, which is the attractor of the iIFS $\{f_n\}_{n\geq1}$ corresponds to a sequence of integers $\{a_n\}_{n\geq1}$ in the sense that
\[
x=\lim\limits_{n\rightarrow\infty}f_{a_{1}}\circ\cdots\circ f_{a_{n}}(1).
\]
We call $\{a_n\}_{n\geq1}$ the digits of $x$. It should be pointed out that the sequence of digits of one point may not be unique. However, at most a countable number of points  can be ignored as the Hausdorff dimension is concerned, then there is a 1-to-1 correspondence between a real number in $[0,1]$ and a sequence of integers. When the sequence of digits is unique, we write $x=(a_1(x),a_2(x),\ldots)$ for the symbolic expansion of any point $x\in[0,1]$. It is known that there are several classical iIFS closely connected with number theory:
\begin{itemize}
\item[$\bullet$]Continued fraction system:
\[
f_n(x)=\frac{1}{x+n},\ \ x\in[0,1],\ n\in\mathbb{N}.
\]
The sequence of digits $\{a_n(x)\}_{n\geq1}$ is just the partial quotients of $x$ in its continued fraction expansion.

\smallskip

\item[$\bullet$]L\"{u}roth system:
\[
f_n(x)=\frac{x}{n(n+1)}+\frac{1}{n+1},\ \ x\in[0,1],\ n\in\mathbb{N}.
\]
Then $\{a_n(x)\}_{n\geq1}$ is just the sequence of digits in the L\"{u}roth series expansion of $x$.

\item[$\bullet$] Quadratic Gauss system:
\[
f_n(x)=\frac{1}{(x+n)^2},\ \ x\in[0,1],\ n\in\mathbb{N}.
\]
Here the sequence $\{a_n(x)\}_{n\geq1}$ is just the digits in the series expansion of $x$ induced by the quadratic Gauss map, which is a particular case of $f$-expansion defined by R\'{e}nyi \cite{R57}. For more details on the quadratic Gauss map, see \cite[Section 1]{GHST23}.
\end{itemize}
Both continued fraction system and L\"{u}roth system are special 2-decaying Gauss like iIFS, and the quadratic Gauss system is a special 3-decaying Gauss like iIFS.

\medskip

For any $(a_{1},\ldots,a_{n})\in\mathbb{N}^{n}$,
\[
I_{n}(a_{1},\ldots,a_{n})=f_{a_{1}}\circ\cdots\circ f_{a_{n}}([0,1])
\]
is called an \emph{$n$th-level cylinder} denoting the set of points in [0,1] whose symbolic expansions begin with $a_{1},\ldots,a_{n}$. Notice that the conditions (\romannumeral2) and (\romannumeral3) in the definition of the $d$-decaying Gauss like iIFS can be used to estimate the upper and lower bounds for these cylinders. In what follows, we always assume without loss of generality that $\varepsilon=0$ in (\romannumeral3) for simplicity. For the general case, in fact, it suffices to replace $d$ by $d+\varepsilon$ for the lower bound and by $d-\varepsilon$ for the upper bound, and then let $\varepsilon\rightarrow0$. Thus we have
\begin{equation}\label{lens}
C_{1}^{n}\prod\limits_{i=1}^{n}a_{i}^{-d}\leq|I_{n}(a_{1},\ldots a_{n})|\leq C_{2}^{n}\prod\limits_{i=1}^{n}a_{i}^{-d}.
\end{equation}

\medskip

Determining the fractal dimensions of such symbolic expansions that concerns the properties of digits has always been an important subject in the study of $d$-decaying Gauss like iIFS, see \cite{JR,CWW,LR21,Zhang20,GHST23} and references therein. Among them, Jordan and Rams \cite{JR} considered the dimension of the sets of points with strictly increasing digits in general $d$-decaying iIFS and obtained the following theorem.
\begin{thm}\label{JR thm}\textbf{\rm(\cite{JR})}
Let $\Phi:\mathbb{N}\rightarrow\mathbb{R}$ be a function such that $n\leq\Phi(n)\leq\beta n$ for some $\beta\geq1$, then
\[
\dim_{\rm H}\Pi\{\underline{a}:\ a_{n+1}>\Phi(a_n)\ \text{for all}\ n\in\mathbb{N}\}=1/d,
\]
where we use $\dim_{\rm H}$ to denote the Hausdorff dimension.
\end{thm}
\noindent Besides, Jordan and Rams also showed that the dimension does drop below $1/d$ when the digits grow with exponential rates in the $d$-decaying Gauss like iIFS. After that, Cao, Wang and Wu \cite{CWW} investigated the digits of points which are further restricted to an infinite subset of positive integers in the iIFS with some general regular properties, including $d$-decaying Gauss like. They also obtained the dimensions of sets of points satisfying much larger growth rate of the digits. The Hausdorff dimension of sets of points with restricted slowly growing digits in such iIFS was studied shortly by Gonz\'{a}lez-Robert, Hussain, Shulga and Takahasi \cite{GHST23}. Recently, Liao and Rams \cite{LR21} concerned the increasing rate of Birkhoff sums in such systems and calculated the Hausdorff dimensions of sets of points whose Birkhoff sums share the same increasing rate for different unbounded potential functions. Later, Zhang \cite{Zhang20} studied the critical cases for the growth rate functions which are not discussed in \cite{LR21}, and pointed out that the corresponding Hausdorff dimension spectrum is right continuous.

\smallskip

Before proceeding, we shall state a fact that there exists a conformal measure (here is the $1$-dimensional Lebesgue measure) which is equivalent to the unique ergodic invariant measure in $d$-decaying Gauss like iIFS. To this end, in the following we shall recall the classical definitions of confomal iterated function systems and some relevant conclusions coming from \cite{HMU,Mau95,MU96}.

\subsection{Confomal iterated function systems and relevant conclusions}

 For convenience, we give some notations.
\begin{itemize}

\item[$\bullet$] $\phi_{\omega}:=\phi_{\omega_1}\circ\cdots\circ\phi_{\omega_n}$ for $\omega=(\omega_1,\ldots,\omega_n)\in\mathbb{N}^{n},\ n\geq1$.

\medskip

\item[$\bullet$] $\|\phi_{\omega}^{'}\|:=\sup_{x\in X}|\phi_{\omega}^{'}(x)|$ for $\omega\in\cup_{n\geq1}\mathbb{N}^{n}$.

\medskip

\item[$\bullet$] $C(X)$ denotes the space of continuous functions on $X$.

\medskip

\item[$\bullet$] Int$X$ and $\partial X$ denote the interior and the boundary of $X$, respectively.

\medskip

\item[$\bullet$] The $1$-dimensional Lebesgue measure is denoted by $\mathcal{L}$, and the $n$-dimensional Lebesgue measure by $\mathcal{L}^{n}$.
\end{itemize}

\medskip

An iterated function system $S=\{\phi_i\}_{i\in\mathbb{N}}$ is called a \emph{conformal iterated function system (c.i.f.s)} with seed set $X$, where $X$ is a nonempty compact connected subset of $\mathbb{R}^n$, if the following conditions are satisfied.
\begin{itemize}
\item[(1)]For each $i\in\mathbb{N}$, $\phi_i$ is an injective map of $X$ into $X$;

\smallskip

\item[(2)]The system $S$ is uniformly contractive on $X$, namely, there exists $0<s<1$, such that
\[
|\phi_i(x)-\phi_i(y)|\leq s|x-y|;
\]

\smallskip

\item[(3)]The open set condition is satisfied for Int$X$:
\[
\phi_i(\text{Int}X)\subseteq\text{Int}X\ \  \text{and}\ \ \phi_i(\text{Int}X)\cap\phi_j(\text{Int}X)=\emptyset,\ i,j\in\mathbb{N},i\neq j;
\]

\smallskip

\item[(4)]There exists an open connected set $V$ with $X\subseteq V\subseteq \mathbb{R}^n$ such that each $\phi_i$, $i\in\mathbb{N}$, extends to $C^1$ conformal diffeomorphism of $V$ into $V$;

\smallskip

\item[(5)]\emph{Cone condition}: there exists $\gamma,l$ such that for every $x\in\partial X\subseteq\mathbb{R}^n$, there is an open cone Con$(x,\gamma,l)\subset$ Int$X$ with vertex $x$, central angle of Lebesgue measure $\gamma$ and altitude $l$;

\smallskip

\item[(6)]The BDP is satisfied for $S$: there is a $K\geq1$ such that
\[
|\phi_{\omega}^{'}(y)|\leq K|\phi_{\omega}^{'}(x)|
\]
for every $\omega\in\cup_{n\geq1}\mathbb{N}^n$ and every pair of points $x,y\in V$.
\end{itemize}

\begin{lem}
The $d$-decaying Gauss like iIFS is a c.i.f.s.
\end{lem}
\begin{proof}
By the definition of $d$-decaying Gauss like iIFS, the conditions (1)-(4), (6) are obviously satisfied for $S=\{f_i\}_{i\in\mathbb{N}}$ if we take $X=[0,1]$ and $V=(-1,2)$. As for the condition (5), since the Cone condition can be replaced with a weaker condition (see \cite[p.72]{GMW} and \cite[Section 2]{MU96})
\begin{equation}\label{Cone condition}
\inf\limits_{x\in\partial X}\inf\limits_{0<r<1}\mathcal{L}^{n}(B(x,r)\cap\text{Int}X)/\mathcal{L}^{n}(B(x,r))>0,
\end{equation}
where $B(x,r)$ denotes the open ball with center $x$ and radius $r$. In our setting, the condition (\ref{Cone condition}) is clearly satisfied when we let $X=[0,1]$, thus $\mathcal{L}^{n}=\mathcal{L}$, Int$X=(0,1)$ and $\partial X=\{0,1\}$.
\end{proof}

The topological pressure function $P$ for the c.i.f.s $S=\{\phi_i\}_{i\in\mathbb{N}}$ is defined as
\[
P(t):=\lim\limits_{n\rightarrow\infty}\frac{1}{n}\log\sum\limits_{\omega\in\mathbb{N}^n}\|\phi_{\omega}^{'}\|^t.
\]
Detailed properties of this pressure function can be found in \cite{Mau95,MU96,MU03}. As shown in \cite{Mau95}, there are two disjoint classes of c.i.f.s, regular and irregular. A system is called regular if there exists $t\geq0$ such that $P(t)=0$, otherwise the other.

\smallskip

Let $J$ be the attractor of the c.i.f.s $S=\{\phi_i\}_{i\in\mathbb{N}}$, that is $J=\bigcup_{i\in\mathbb{N}}\phi_i(J)$. A Borel probability measure $m$ is said to be $t$-conformal if it is supported on the set $J$, and for every Borel set $A\subseteq X$ and $i,j\in\mathbb{N},\ i\neq j$,
\[
m(\phi_i(A))=\int_{A}|\phi_i^{'}|^tdm \ \ \text{and}\ \ m(\phi_i(X)\cap\phi_j(X))=0.
\]
It is shown in \cite{Mau95} that the c.i.f.s is regular if and only if there exists a $t$-conformal measure ($t$ is such that $P(t)=0$), and then $t=\dim_{\rm{H}}J$.

\medskip

The following lemma provides a way to determine the conformal measure.
\begin{lem} [{\cite[Theorem 7.5]{Mau95}},{\cite[Theorem 4.5]{MU96}}]\label{conformal measure}
If $S$ is a regular c.i.f.s, and $\mathcal{L}(\text{Int}X\setminus X_1)>0$, where $X_1=\bigcup_{i\in\mathbb{N}}\phi_i(X)$, then $\dim_{\rm{H}}J<n$. Conversely, if $\mathcal{L}(X\setminus X_1)=0$, then $S$ is regular, $\mathcal{L}(J)=\mathcal{L}(X)$, and $\mathcal{L}/\mathcal{L}(X)$ is the conformal measure.
\end{lem}

For the regular system, there are at least two ways for deriving an ergodic invariant measure $m^{\ast}$ which is equivalent to the $t$-conformal measure. One of the methods relies on using Banach limits (see \cite[Theorems 8.1 and 8.2]{Mau95} for more details), and the other one depends on applying the Frobenius-Perron operator $L=L_t:\ C(X)\rightarrow C(X)$ defined as follows:
\[
L(f)(x)=\sum\limits_{i\in\mathbb{N}}|\phi_i^{'}(x)|^tf(\phi_i(x)).
\]
\begin{lem}[{\cite[Theorem 8.3]{Mau95}}]\label{existence}
For $m$-a.e. $x\in X$, $\lim\limits_{n\rightarrow\infty}L^n(1)(x)=g(x)$ exists and $g=dm^{\ast}/dm$. In particular, for $m$-a.e. $x\in X$,
\[
g(x)=L(g)(x)=\sum\limits_{i\in\mathbb{N}}|\phi_i^{'}(x)|^tg(\phi_i(x)).
\]

\end{lem}
Based on Lemmas \ref{conformal measure} and \ref{existence}, we immediate obtain the following results.
\begin{thm}\label{dcds}
In $d$-decaying Gauss like iIFS, the $t$-conformal measure $m$ is $1$-dimensional Lebesgue measure $\mathcal{L}$ on $[0,1]$. Moreover, there exists an ergodic invariant probability measure $m^{\ast}$ equivalent to the Lebesgue measure $\mathcal{L}$.
\end{thm}

\subsection{Statements of results}

Much attention has been paid to the convergence exponent for kinds of sequences in number theory and multifractal analysis of various dynamical systems in the past decades. For more details, we refer to \cite{BT,Dod,FLMW10,FMSW21,FMSY24,WW08} and references therein. This paper is mainly concerned with the  multifractal properties of sets concerning the convergence exponent of the sequence of digits in $d$-decaying Gauss like iIFS. The convergence exponent (see P\'{o}lya and Szeg\H{o} \cite[p.26]{PS72}) of the sequence of digits $\{a_n(x)\}_{n\geq1}$  for each $x\in(0,1)$ in $d$-decaying Gauss like iIFS is defined by
\begin{equation}\label{t1}
\tau_1(x):=\inf\left\{s\geq0:\ \sum\limits_{n\geq1}a_n^{-s}(x)<\infty\right\}.
\end{equation}
Applying Birkhoff Ergodic Theorem to $d$-decaying Gauss like iIFS, we deduce from (\ref{lens}) and Theorem \ref{dcds} that for $\mathcal{L}$-a.e $x \in (0,1)$,
\begin{align}\label{erg}
\nonumber\lim\limits_{n\to\infty}\sqrt[n]{a_1(x)\cdots a_n(x)}&=e^{\int_{0}^{1}\log a_1(x)dm^\ast}\\
&\approx e^{\sum\limits_{k\geq1}\int_{I_1(k)}\log a_1(x)dx}\approx e^{\sum\limits_{k\geq1}\frac{\log k}{k^d}}.
\end{align}
Here we write $a\approx b$ for $1/C\leq a/b\leq C$ where $C$ is an absolute constant. Then it is immediate that infinitely many
of $a_n(x)$ is less than some positive constant for $\mathcal{L}$-a.e $x \in (0,1)$. This shows that $\tau_1(x)=\infty$ for $\mathcal{L}$-a.e $x \in (0,1)$. Therefore, it is natural to investigate the sizes of such Lebesgue null sets from the viewpoint of multifractal analysis, that is, the Hausdorff dimension of
\[
E(\alpha):=\left\{x\in(0,1):\ \tau_1(x)=\alpha\right\}, \ \ 0\leq\alpha<\infty.
\]
\begin{thm}\label{dg}
For any $0\leq\alpha<\infty$, we have
\[
\dim_{\rm H}E(\alpha)=1/d.
\]
\end{thm}

\medskip

As we know, the multifractal properties of sets relevant to the growth rate of the product of two consecutive digits has become an emerging and vital subject in view of the  pioneering work of Kleinbock and Wadleigh \cite{KW18}, who considered the improvability of Dirichlet's theorem  in metric Diophantine approximation. In recent years, there are abundant relevant results in continued fraction systems, see \cite{BF,FX,HHY,HW,HWX,HKWW,SZ2024} for more details. Motivated by this, we also wonder the multifractal properties of the convergence exponent which are relevant to the growth rate of weighted products of distinct digits with finite numbers in $d$-decaying Gauss like iIFS. To be precise,  we first consider the convergence exponent defined as
\begin{equation}\label{t2}
\tau_2(x):=\inf\left\{s\geq0:\ \sum\limits_{n\geq1}\left(a_n^{t_0}(x)\cdots a_{n+m}^{t_m}(x) \right)^{-s}<\infty\right\},
\end{equation}
where $m\geq1$ and the weights $\{t_i\}_{0\leq i\leq m}$ is a sequence of non-negative real numbers. Without loss of generality, we assume  that $t_0\neq0$ and at least one $t_i\neq0\ (1\leq i\leq m)$ here and in the sequel. Notice that the assumption is valid. Indeed, for the case when the weights $\{t_i\}_{0\leq i\leq m}$ is a sequence of positive real numbers, the convergence exponent $\tau_2(x)$ is defined by the weighted products of consecutive digits, while for some $t_i=0\ (1\leq i\leq m)$, it is defined by the weighted products of lacunary digits.
It follows from \eqref{erg} that $\tau_2(x)=\infty$ for $\mathcal{L}$-a.e $x \in (0,1)$. Then we are also interested in the Hausdorff dimension of Lebesgue null sets
\[
E(\alpha,\{t_i\}_{0\leq i\leq m}):=\left\{x\in(0,1):\ \tau_2(x)=\alpha\right\}, \ \ 0\leq\alpha<\infty.
\]
\begin{thm}\label{ddg}
For any $0\leq\alpha<\infty$, we have
\[
\dim_{\rm H}E(\alpha,\{t_i\}_{0\leq i\leq m})=1/d.
\]
\end{thm}
\medskip

Compared with the results in Theorems \ref{dg} and \ref{ddg}, we know that the Hausdorff dimensions of the corresponding level sets do not vary by changing the number of the digits and the weights. Thus, it is natural to consider the cases when the Hausdorff dimensions of the above level sets are rely on the weights $\{t_i\}_{0\leq i\leq m}$. Note that if the sequence of digits $\{a_n(x)\}_{n\geq1}$ is non-decreasing, then the sequence $\{a^{t_0}_n(x)\cdots a^{t_m}_{n+m}(x)\}_{n\geq1}$ is also non-decreasing, by a result of \cite[Lemma 2.1]{WW08},
the convergence exponent defined by \eqref{t2} can be written as
\begin{equation}\label{tauzj}
\tau_{2}(x)=\limsup\limits_{n\to\infty}\frac{\log n}{\log\big(a_n^{t_0}(x)\cdots a_{n+m}^{t_m}(x)\big)}.
\end{equation}
In the following we continue to study the multifractal spectrum of $\tau_{2}(x)$ defined by \eqref{tauzj},
i.e., the Hausdorff dimension of the intersection of sets $E(\beta,\{t_i\}_{0\leq i\leq m})$
and $\Lambda$ for any $0\leq\beta\leq\infty$,
where
\[\Lambda=\big\{x\in(0,1):\ a_{n}(x)\leq a_{n+1}(x),\ \forall\ n\geq1\big\}.\]
It can be concluded by Theorem \ref{JR thm} that $\dim_{\rm H}\Lambda=1/d$. Let $\alpha=1/\beta$. Then we deduce from \eqref{tauzj} that
\begin{align*}
E(\Lambda,\alpha,\{t_i\}_{0\leq i\leq m})&:=E(\beta,\{t_i\}_{0\leq i\leq m})\cap\Lambda\\
&=\left\{x\in\Lambda: \liminf\limits_{n\to\infty}\frac{\log\left(a_n^{t_0}(x)\cdots a_{n+m}^{t_m}(x)\right)}{\log n}=\alpha\right\}.
\end{align*}
In the following we shall show that the Hausdorff dimension of the level sets $E(\Lambda,\alpha,\{t_i\}_{0\leq i\leq m})$ are closely depended on the level $\alpha$ and the weights $\{t_i\}_{0\leq i\leq m}$. Now we are in a position to state the results.
\begin{thm}\label{jqdg}
Let $\Sigma_t=\sum_{0\leq i\leq m}t_i$. Then for any $0\leq\alpha\leq\infty$, we have
\begin{equation*}
\dim_{\rm H}E(\Lambda,\alpha,\{t_i\}_{0\leq i\leq m})=
\begin{cases}
0,\ \ \ \ \ \ \ \ \ \ \ \ 0\leq\alpha<\Sigma_t,\cr
\frac{\alpha-\Sigma_t}{d\alpha},\ \ \ \ \ \ \ \Sigma_t\leq\alpha<\infty;\cr
\frac{1}{d},\ \ \ \ \ \ \ \ \ \ \ \ \ \ \ \ \ \ \alpha=\infty.
\end{cases}
\end{equation*}
\end{thm}
Replacing the lower limit by limit or upper limit in the set $E(\Lambda,\alpha,\{t_i\}_{0\leq i\leq m})$, we are also interested in the Hausdorff dimension of the sets
\[
F(\Lambda,\alpha,\{t_i\}_{0\leq i\leq m})=\left\{x\in\Lambda: \lim\limits_{n\to\infty}\frac{\log\left(a_n^{t_0}(x)\cdots a_{n+m}^{t_m}(x)\right)}{\log n}=\alpha\right\}
\]
and
\[
G(\Lambda,\alpha,\{t_i\}_{0\leq i\leq m})=\left\{x\in\Lambda: \limsup\limits_{n\to\infty}\frac{\log\left(a_n^{t_0}(x)\cdots a_{n+m}^{t_m}(x)\right)}{\log n}=\alpha\right\}.
\]
\begin{thm}\label{supp}
For any $0\leq\alpha\leq\infty$, we have
\[
\dim_{\rm H}F(\Lambda,\alpha,\{t_i\}_{0\leq i\leq m})=\dim_{\rm H}G(\Lambda,\alpha,\{t_i\}_{0\leq i\leq m})=\dim_{\rm H}E(\Lambda,\alpha,\{t_i\}_{0\leq i\leq m}).
\]
\end{thm}

\medskip

Throughout this paper, we use $\mathcal{H}^{s}$ to denote the $s$-dimensional Hausdorff measure of a set, $\lfloor x\rfloor$ the largest integer not exceeding $x$ and $\sharp$ the cardinality of a set, respectively. The paper is organized as follows. In section 2, we collect and establish some dimensional results in $d$-decaying Gauss like iIFS. Section 3 is devoted to the proofs of the main results.

\smallskip

\section{Some useful lemmas}
In this section, we present some useful lemmas for calculating the Hausdorff dimension of certain sets in $d$-decaying Gauss like iIFS. The first lemma is to deal with the Hausdorff dimension of Good's type sets \cite[Theorem 1]{Good41} of points whose digits tend to infinity without any restriction. Let
\[E_1=\big\{x\in(0,1):\ a_n(x)\to\infty\ \text{as}\ n\to\infty\big\}.\]
\begin{lem}\label{GT}
Let $\{a_n(x)\}_{n\geq1}$ be the digits in $d$-decaying Gauss like iIFS.
Then we have $\dim_{\rm H}E_1=1/d$.
\end{lem}
\begin{proof}
The lower bound estimation can be deduced from Theorem \ref{JR thm} given by Jordan and Rams \cite{JR}, where the dimension of the set of points  with strictly increasing digits in general $d$-decaying iIFS is obtained. In what follows, it suffices to give the upper bound estimation for $\dim_{\rm H}E_1$.

\medskip

Fix $M\in\mathbb{N}$. For any $x\in E_1$, there exists $N\in\mathbb{N}$, such that $a_n(x)\geq M$ for each $n>N$. Then
\begin{align*}
E_1&\subseteq\bigcap\limits_{M=1}^{\infty}\bigcup\limits_{N=1}^{\infty}\big\{x\in(0,1):\ a_n(x)\geq M,\ \forall\ n>N\big\}\\
&\subseteq\bigcap\limits_{M=1}^{\infty}\bigcup\limits_{N=1}^{\infty}\bigcup\limits_{a_1,\cdots,a_n\in\mathbb{N}}F_M(N),
\end{align*}
where for each $a_1,\ldots,a_n\in\mathbb{N}$, we write
\[
F_M(N):=\big\{x\in(0,1):\ a_k(x)=a_k,\ 1\leq k\leq N,\ a_n(x)\geq M,\ \forall\ n>N\big\}.
\]
Clearly, the monotonicity and countable stability properties of the Hausdorff dimension (see \cite[p. 32]{Fal90}) imply that
\[
\dim_{\rm H}E_1\leq\sup_{N\geq1}\{\dim_{\rm H}F_{M}(N)\}.
\]
Thus for any $s>1/d$, we have
\begin{align*}
\mathcal{H}^{s}(F_M(N))&\leq\liminf\limits_{n\to\infty}\sum\limits_{a_{k}\geq M,\ N<k\leq n}|I_{n}(a_1, \ldots, a_n)|^s\\
&\leq\liminf\limits_{n\to\infty}\sum\limits_{a_{k}\geq M,\ N<k\leq n}\prod\limits_{k=1}^n\left(\frac{C_2}{a_k^d}\right)^s\\
&\leq\prod\limits_{k=1}^N\left(\frac{C_2}{a_k^d}\right)^s\cdot\liminf\limits_{n\to\infty}\left(\sum\limits_{b\geq M}\frac{C_2^s}{b^{ds}}\right)^{n-N}.
\end{align*}
By the choice of $s$, we can choose $M_0$ sufficiently large ensuring that
\[
\sum\limits_{b\geq M_0}\frac{C_2^s}{b^{ds}}\leq1.
\]
Hence when $M\geq M_0$, we have
\[
\mathcal{H}^{s}(F_M(N))\leq\prod\limits_{k=1}^N\left(\frac{C_2}{a_k^d}\right)^s<\infty,
\]
which means, when $M\geq M_0$,
\[
\dim_{\rm H}F_M(N)\leq\frac{1}{d}
\]
and this establishes the result.
\end{proof}

\medskip

The following lemma is concerned with the Hausdorff dimension of the set of points whose product of consecutive digits tends to infinity without any restriction in $d$-decaying Gauss like iIFS. Let
\[E_2=\big\{x\in(0,1):\ a_n(x)\cdots a_{n+m}(x)\to\infty\ \text{as}\ n\to\infty\big\}.\]
\begin{lem}\label{LC}
For any $m\geq1$, we have $\dim_{\rm H}E_2=1/d$.
\end{lem}
\begin{proof}
It is obvious that $E_1\subseteq E_2$, and thus by Lemma \ref{GT},
\[\dim_{\rm H}E_2\geq\dim_{\rm H}E_1=1/d.\]
To bound $\dim_{\rm H}E_2$ from upper, we divide the proof into three steps. First, for any $M\geq1$, we deduce that
\begin{align*}
E_2&\subseteq\big\{x\in(0,1):\ a_n(x)\cdots a_{n+m}(x)\geq M,\ \text{for}\ n\ \text{ultimately}\big\}\\
&=\bigcup_{N\geq1}\big\{x\in(0,1):\ a_n(x)\cdots a_{n+m}(x)\geq M,\ \forall\ n\geq N\big\}:=\bigcup_{N\geq1}E_{M}(N).
\end{align*}
From a result of Good \cite[Lemma 1]{Good41}, we know that for any $N\geq1$, $\dim_{\rm H}E_{M}(N)=\dim_{\rm H}E_{M}(1)$. Then we have
\[\dim_{\rm H}E_2\leq\sup_{N\geq1}\{\dim_{\rm H}E_{M}(N)\}.\]
 Thus it only needs to estimate the upper bound of $\dim_{\rm H}E_{M}(1)$. Let
\[\mathcal{C}_{n}(M)=\big\{(a_{1},\ldots,a_{n})\in\mathbb{N}^{n}:\ a_k(x)\cdots a_{k+m}(x)\geq M,\ 1\leq k\leq n-m\big\}.\]
Then we have
\begin{equation}\label{bgg}
E_{M}(1)\subseteq\bigcap_{n\geq1}\bigcup_{(a_{1},\ldots,a_{n})\in \mathcal{C}_{n}(M)}I_n(a_1,\ldots,a_n).
\end{equation}
Secondly, for any $t>1$ and $(a_{1},\ldots,a_{n})\in\mathbb{N}^{n}$, we define a family of probability measures $\{\mu_t\}_{t>1}$ on $n$th-level cylinder $I_n(a_1,\ldots,a_n)$ such that
\begin{equation}\label{mut}
\mu_{t}(I_n(a_1,\ldots,a_n))=e^{-n\log\zeta(t)-t\sum_{1\leq j\leq n}\log a_j},
\end{equation}
where $\zeta(t)=\sum_{k\geq1}k^{-t}$. It is easy to verify that
\[\sum_{(a_{1},\ldots,a_{n})\in\mathbb{N}^{n}}\mu_{t}(I_n(a_1,\ldots,a_n))=1\]
and
\[\sum_{a_{n+1}\in\mathbb{N}}\mu_{t}(I_{n+1}(a_1,\ldots,a_n,a_{n+1}))=\mu_{t}(I_n(a_1,\ldots,a_n)).\]
Thus by Hahn-Kolmogorov extension theorem, we know that the probability measures defined by \eqref{mut}
can be extended on $[0,1]$. Thirdly, for any $\varepsilon>0$, choose $s=(t+\varepsilon)/d$ and $M$ large enough such that
\begin{equation}\label{tj}
n\frac{t+\varepsilon}{d}\log C_2+n\log\zeta(t)<\varepsilon\frac{n-m}{m+1}\log M.
\end{equation}
Now it turns to estimate the $s$-dimensional Hausdorff measure of $E_M(1)$. Notice that for any $(a_{1},\ldots,a_{n})\in \mathcal{C}_{n}(M)$,
\[\sum_{1\leq i\leq n}\log a_i=\frac{1}{m+1}\log(a_1\cdots a_n)^{m+1}\geq\frac{n-m}{m+1}\log M,\]
which, in combination with \eqref{lens}, \eqref{bgg}, \eqref{mut} and \eqref{tj}, implies that
\begin{align*}
\mathcal{H}^{s}(E_M(1))&\leq\liminf\limits_{n\to\infty}\sum\limits_{(a_{1},\ldots,a_{n})\in \mathcal{C}_{n}(M)}|I_{n}(a_1, \ldots, a_n)|^s\\
&\leq\liminf\limits_{n\to\infty}\sum\limits_{\substack{(a_1,\ldots,a_n)\in\mathbb{N}^n\\a_ka_{k+1}\cdots a_{k+m}\geq M\\1\leq k\leq n-m}}e^{\frac{n(t+\varepsilon)\log C_2}{d}-\varepsilon\sum_{1\leq i\leq n}\log a_i-t\sum_{1\leq i\leq n}\log a_i}\\
&\leq\liminf\limits_{n\to\infty}\sum\limits_{(a_1,\ldots,a_n)\in\mathbb{N}^n}e^{\frac{n(t+\varepsilon)\log C_2}{d}-\varepsilon\frac{n-m}{m+1}\log M-t\sum_{1\leq i\leq n}\log a_i}\\
&\leq\liminf\limits_{n\to\infty}\sum\limits_{(a_1,\ldots,a_n)\in\mathbb{N}^n}e^{-n\log\zeta(t)-t\sum_{1\leq i\leq n}\log a_i}\\
&=\liminf\limits_{n\to\infty}\sum\limits_{(a_1,\ldots,a_n)\in\mathbb{N}^n}\mu_t(I_n(a_1,\ldots,a_n))=1.
\end{align*}
This shows that $\dim_{\rm H}E_{M}(1)\leq s$. Letting $t\to1$ and $\varepsilon\to0$, we obtain the desired upper bound of $\dim_{\rm H}E_2$, i.e.,
\[\dim_{\rm H}E_2\leq\dim_{\rm H}E_{M}(1)\leq\frac{1}{d}.\]
\end{proof}
\indent For the case when sets of points whose digits tend to infinity with various growth rates, their Hausdorff dimensions could be determined by the following lemma.
\begin{lem}[{\cite[Lemma 2.3]{LR21}}]\label{sbk}
 Let $\{s_n\}_{n\geq1}, \{r_n\}_{n\geq1}$ be two sequences of positive real numbers and for any $N\geq1$, let
\[B(\{s_n\},\{r_n\},N)=\big\{x\in(0,1):\ s_n-r_n\leq a_n(x)\leq s_n+r_n,\ \forall\ n\geq N\big\}.\]
If $s_n>r_n$ for any $n\geq1$, $s_n,r_n\to\infty$ as $n\to\infty$, and \[\liminf\limits_{n\to\infty}\frac{s_n-r_n}{s_n}>0.\]
Then we have
\[\dim_{\rm H}B(\{s_n\},\{r_n\},N)=\liminf\limits_{n\to\infty}\frac{\sum_{1\leq i\leq n}\log r_i}
{d\sum_{1\leq i\leq n+1}\log s_i-\log r_{n+1}}.\]
\end{lem}
We note that
\[
\dim_{\rm H}B(\{s_{n}\},\{r_{n}\},N)=\dim_{\rm H}B(\{s_{n}\},\{r_{n}\},1).
\]
To see this, it suffices to notice that the dimensional formula in Lemma \ref{sbk} does not depend on a finite number of initial terms of the sequences $\{s_n\}$ and $\{r_n\}$. Besides, the set $B(\{s_{n}\},\{r_{n}\},N)$ can be represented as a countable union of bi-Lipschitz images of $B(\{s_{n+N-1}\},\{r_{n+N-1}\},1)$, and it is known that bi-Lipschitz maps always preserve the Hausdorff dimension.

\medskip

To end this section, we present a combinatorial formula for calculating the cardinality of some finite sets of points whose digits are non-decreasing.
\begin{lem}[{\cite[Lemma 2.5]{FMSY24}}]\label{card}
For any positive integers $\ell$ and $n$, we have
\[
\sharp\big\{(a_1,\ldots,a_n)\in\mathbb{N}^{n}: 1\leq a_1\leq\cdots \leq a_n\leq\ell\big\}=\frac{(n+\ell-1)!}{n!(\ell-1)!}.
\]
\end{lem}

\smallskip

\section{Proofs of the main results}
This section is devoted to the proofs of the main results. We will divide them into three parts.
\subsection{Proof of Theorem \ref{dg}}
Let $0\leq\alpha<\infty$. From \eqref{t1}, we know that for any $x\in E(\alpha)$ and $\varepsilon>0$, $\sum_{n=1}^{\infty}(a_n(x))^{-(\alpha+\varepsilon)} <\infty$. Then we have
$E(\alpha)\subseteq E_1$, and thus by Lemma \ref{GT}, \[\dim_{\rm H}E(\alpha)\leq\dim_{\rm H}E_1=1/d.\] To bound $\dim_{\rm H}E(\alpha)$ from below, we need to construct suitable Cantor-type subset of $\dim_{\rm H}E(\alpha)$ according $\alpha=0$ and $\alpha\in(0,\infty)$.
\begin{enumerate}[(i)]
\item For the case $\alpha=0$, let
\[\{s_n\}_{n\geq1}=\{2e^{n}\}_{n\geq1}\ \text{and}\  \{r_n\}_{n\geq1}=\{e^{n}\}_{n\geq1}.\]
Then it is easy to check that
\begin{equation}\label{Bzj}
B(\{s_n\},\{r_n\},1)\subseteq E(0).
\end{equation}
Appying Lemma \ref{sbk}, we immediately obtain that
\[\dim_{\rm H}E(0)\geq \dim_{\rm H}B(\{s_n\},\{r_n\},1)=\liminf\limits_{n\to\infty}\frac{\sum_{1\leq i\leq n}\log r_i}
{d\sum_{1\leq i\leq n+1}\log s_i-\log r_{n+1}}=\frac{1}{d}.\]
\item For the case $\alpha\in(0,\infty)$, let
\[\{s_n\}_{n\geq1}=\{2n^{1/\alpha}\}_{n\geq1}\ \text{and}\  \{r_n\}_{n\geq1}=\{n^{1/\alpha}\}_{n\geq1}.\]
Then by Lemma \ref{sbk} again, we have
    \[\dim_{\rm H}E(\alpha)\geq \dim_{\rm H}B(\{s_n\},\{r_n\},1)=\frac{1}{d}.\]
\end{enumerate}

\medskip

\subsection{Proof of Theorem \ref{ddg}} Recall that
\[
E(\alpha,\{t_i\}_{0\leq i\leq m}):=\left\{x\in(0,1):\ \tau_2(x)=\alpha\right\}, \ \ 0\leq\alpha<\infty.
\]
 With the same method used for estmating the upper bound of $\dim_{\rm H}E(\alpha)$ in the proof of Theorem \ref{dg}, we have
\begin{align*}
E(\alpha,\{t_i\}_{0\leq i\leq m})
&\subseteq\big\{x\in(0,1):\ a^{t_0}_n(x)\cdots a^{t_m}_{n+m}(x)\to\infty\ \text{as}\ n\to\infty\big\}\\
&\subseteq\big\{x\in(0,1):\ a_n(x)\cdots a_{n+m}(x)\to\infty\ \text{as}\ n\to\infty\big\}=E_2.
\end{align*}
It follows from Lemma \ref{LC} that
\[\dim_{\rm H}E(\alpha,\{t_i\}_{0\leq i\leq m})\leq\dim_{\rm H}E_2\leq\frac{1}{d}.\]
For the lower bound of $\dim_{\rm H}E(\alpha,\{t_i\}_{0\leq i\leq m})$, we divide the proof into two parts. It is worth pointing out that the subset defined in \eqref{Bzj} is also suitable for $E(0,\{t_i\}_{0\leq i\leq m})$, and then $\dim_{\rm H}E(0,\{t_i\}_{0\leq i\leq m})\geq1/d$.
Thus it remains to construct a subset of $E(\alpha,\{t_i\}_{0\leq i\leq m})$ for any $\alpha\in(0,\infty)$. Let
\[
\{s_n\}_{n\geq1}=\left\{2n^{\frac{1}{\alpha\Sigma_t}}\right\}_{n\geq1}\ \text{and}\ \{r_n\}_{n\geq1}=\left\{n^{\frac{1}{\alpha\Sigma_t}}\right\}_{n\geq1}.
\]
Then the definition of the convergence exponent $\tau_2(x)$ in \eqref{t2} shows that
\[B(\{s_n\},\{r_n\},1)\subseteq E(\alpha,\{t_i\}_{0\leq i\leq m}).\] We conclude from Lemma \ref{sbk} that
\[\dim_{\rm H}E(\alpha,\{t_i\}_{0\leq i\leq m})\geq\dim_{\rm H}B(\{s_n\},\{r_n\},1)=\frac{1}{d}.\]

\medskip

\subsection{Proof of Theorem \ref{jqdg}} In the following we shall deal with the proofs of Theorem \ref{jqdg} into three cases. Recall that
\[E(\Lambda,\alpha,\{t_i\}_{0\leq i\leq m})=\left\{x\in\Lambda: \liminf\limits_{n\to\infty}\frac{\log\left(a_n^{t_0}(x)\cdots a_{n+m}^{t_m}(x)\right)}{\log n}=\alpha\right\}.\]
\subsubsection{\textbf{Case} $0\leq\alpha<\Sigma_t$} For any $s\in(0,\Sigma_t-\alpha)$, let
\[
\mathcal{D}_{k}(\alpha,\{t_i\}_{0\leq i\leq m}):=\left\{(a_1, \ldots, a_{k})\in\mathbb{N}^{k}: 1\leq a_1\leq\cdots\leq a_k\leq k^{\frac{\alpha+s}{\Sigma_t}}\right\},
\]
then
\begin{align}\label{fg1}
\nonumber E(\Lambda,\alpha,\{t_i\}_{0\leq i\leq m})&\subseteq\bigcap\limits_{n\geq1}\bigcup\limits_{k\geq n}\Big\{x\in\Lambda: a^{t_0}_k(x)\cdots a^{t_m}_{k+m}(x)\leq k^{\alpha+s}\Big\}\\
\nonumber&\subseteq\bigcap\limits_{n\geq1}\bigcup\limits_{k\geq n}\Big\{x\in\Lambda: a_k(x)\leq k^{\frac{\alpha+s}{\Sigma_t}}\Big\}\\
&=\bigcap\limits_{n\geq1}\bigcup\limits_{k\geq n}\bigcup\limits_{(a_1, \ldots, a_{k})\in\mathcal{D}_{k}(\alpha,\{t_i\}_{0\leq i\leq m})}I_k(a_1, \ldots, a_{k}).
\end{align}
Next our aim is to show that $\mathcal{H}^{s}(E(\Lambda,\alpha,\{t_i\}_{0\leq i\leq m}))=0$. For this purpose,
we shall estimate the cardinality of the set $\mathcal{D}_{k}(\alpha,\{t_i\}_{0\leq i\leq m})$ and the length of
$k$th-level cylinder $I_k(a_1,\ldots,a_k)$ for any $x\in\mathcal{D}_{k}(\alpha,\{t_i\}_{0\leq i\leq m})$.
By Lemma \ref{card},
\begin{align}\label{ges1}
\nonumber\sharp\mathcal{D}_{k}(\alpha,\{t_i\}_{0\leq i\leq m})
&=\frac{(k+\lfloor k^{\frac{\alpha+s}{\Sigma_t}}\rfloor-1)!}{k!(\lfloor k^{\frac{\alpha+s}{\Sigma_t}}\rfloor-1)!}\\
\nonumber&\leq(k+1)\cdot(k+2)\cdots(k+\lfloor k^{\frac{\alpha+s}{\Sigma_t}}\rfloor-1)\\
&\leq(k+k^{\frac{\alpha+s}{\Sigma_t}})^{k^{\frac{\alpha+s}{\Sigma_t}}}\leq e^{(1+\log k)k^{\frac{\alpha+s}{\Sigma_t}}}.
\end{align}
On the other hand, for any $x\in\mathcal{D}_{k}(\alpha,\{t_i\}_{0\leq i\leq m})$, there exsits $k_0$ such that for any $k>k_0$, we have $a_k\geq C_2+1$. Then it follows from \eqref{lens} that
\begin{equation}\label{cdgj1}
|I_k(a_1,\ldots,a_k)|\leq(C_2)^k\prod_{1\leq i\leq k}a_i^{-d}\leq(C_2+1)^{dk_0}\Big(\frac{C_2}{(C_2+1)^d}\Big)^k.
\end{equation}
Together with \eqref{fg1}, \eqref{ges1} and \eqref{cdgj1}, we deduce that
\begin{align*}
&\ \mathcal{H}^{s}(E(\Lambda,\alpha,\{t_i\}_{0\leq i\leq m}))\\
&\leq\liminf\limits_{n\to\infty}\sum\limits_{k\geq n}\sum\limits_{(a_1, \ldots, a_{k})\in\mathcal{D}_{k}(\alpha,\{t_i\}_{0\leq i\leq m})}|I_{k}(a_1, \ldots, a_k)|^s\\
&\leq\liminf\limits_{n\to\infty}\sum\limits_{k\geq n}\big(\sharp\mathcal{D}_{k}(\alpha,\{t_i\}_{0\leq i\leq m})\cdot|I_{k}(a_1, \ldots, a_k)|^s\big)\\
&\leq(C_2+1)^{sdk_0}\liminf\limits_{n\to\infty}\sum\limits_{k\geq n}e^{(1+\log k)k^{\frac{\alpha+s}{\Sigma_t}}}\cdot \Big(\frac{C_2}{(C_2+1)^d}\Big)^{ks}=0.
\end{align*}
\begin{rem}\label{45}
Let $0\leq\alpha<\Sigma_t$. Without significant modifications, we also have
\[
\dim_{\rm H}\left\{x\in\Lambda: \liminf\limits_{n\to\infty}\frac{\log\left(a_n^{t_0}(x)\cdots a_{n+m}^{t_m}(x)\right)}{\log n}\leq\alpha\right\}=0.
\]
\end{rem}

\subsubsection{\textbf{Case} $\Sigma_{t}\leq\alpha<\infty$} To bound $\dim_{\rm H}E(\Lambda,\alpha,\{t_i\}_{0\leq i\leq m})$ from upper, our strategy is to find a natural cover by using the construction of $E(\Lambda,\alpha,\{t_i\}_{0\leq i\leq m})$, while for the lower bound of $\dim_{\rm H}E(\Lambda,\alpha,\{t_i\}_{0\leq i\leq m})$, we need to consruct a Cantor-type subset of $E(\Lambda,\alpha,\{t_i\}_{0\leq i\leq m})$.

\medskip

\textbf{Upper bound:}
Let $0<\varepsilon<\alpha$.
 Given a point $x\in E(\Lambda,\alpha,\{t_i\}_{0\leq i\leq m})$, we know that $x\in\Lambda$,
and there exists $N\geq1$ such that $a^{t_0}_j(x)\cdots a^{t_m}_{j+m}(x)\geq j^{\alpha-\varepsilon}$
for any $j\geq N$ and $a^{t_0}_k(x)\cdots a^{t_m}_{k+m}(x)\leq k^{\alpha+\varepsilon}$ for infinitely many $k$'s. Let $B_{N}(\alpha,\{t_i\}_{0\leq i\leq m})$ be the set
\[
\bigcap\limits_{n\geq N}\bigcup\limits_{k\geq n}
\left\{x\in\Lambda: a_{j}(x)\geq (j-m)^{\frac{\alpha-\varepsilon}{\Sigma_t}},
a_k(x)\leq k^{\frac{\alpha+\varepsilon}{\Sigma_t}},\ \forall\ N\leq j\leq k\right\}.
\]
Then we have \[E(\Lambda,\alpha,\{t_i\}_{0\leq i\leq m})\subseteq\bigcup_{N\geq1}B_{N}(\alpha,\{t_i\}_{0\leq i\leq m}).\]
Thus it follows that
\begin{equation}\label{wsb}
\dim_{\rm H}E(\Lambda,\alpha,\{t_i\}_{0\leq i\leq m})\leq\sup\limits_{N\geq1}\Big\{\dim_{\rm H}B_{N}(\alpha,\{t_i\}_{0\leq i\leq m})\Big\}.
\end{equation}
Next we only estimate the upper bound of $\dim_{\rm H}B_{1}(\alpha,\{t_i\}_{0\leq i\leq m})$,
since the other cases are similar. Denoted by
\begin{equation}\label{b1}
B_{1}(\alpha,\{t_i\}_{0\leq i\leq m})=\bigcap\limits_{n\geq 1}\bigcup\limits_{k\geq n}\bigcup\limits_{(a_1, \ldots, a_{k})\in\widetilde{\mathcal{D}_{k}}(\alpha,\{t_i\}_{0\leq i\leq m})}I_k(a_1, \ldots, a_{k}),
\end{equation}
where
\begin{align*}
\widetilde{\mathcal{D}_{k}}(\alpha,\{t_i\}_{0\leq i\leq m})=\Big\{(a_1, \ldots, a_{k})\in\mathbb{N}^{k}:&\ 1\leq a_1\leq \cdots\leq a_k\leq k^{\frac{\alpha+\varepsilon}{\Sigma_t}},\\
&a_{j}(x)\geq (j-m)^{\frac{\alpha-\varepsilon}{\Sigma_t}},\ \forall\ 1\leq j\leq k\Big\}.
\end{align*}
For any $(a_1, \ldots, a_{k})\in\widetilde{\mathcal{D}_{k}}(\alpha,\{t_i\}_{0\leq i\leq m})$, the Stirling formula:
\[\sqrt{2\pi}n^{n+\frac{1}{2}}e^{-n}\leq n!\leq en^{n+\frac{1}{2}}e^{-n}\]
and Lemma \ref{card} deduce that
\begin{align}\label{ges2}
\nonumber\sharp\widetilde{\mathcal{D}_{k}}(\alpha,\{t_i\}_{0\leq i\leq m})&\leq\frac{(k+\lfloor k^{\frac{\alpha+\varepsilon}{\Sigma_t}}\rfloor-1)!}{k!(\lfloor k^{\frac{\alpha+\varepsilon}{\Sigma_t}}\rfloor-1)!}\\
\nonumber&=\frac{\lfloor k^{\frac{\alpha+\varepsilon}{\Sigma_t}}\rfloor\cdot(\lfloor k^{\frac{\alpha+\varepsilon}{\Sigma_t}}\rfloor+1)\cdots(\lfloor k^{\frac{\alpha+\varepsilon}{\Sigma_t}}\rfloor+k-1)}{k!}\\
\nonumber&\leq\frac{k^{k\cdot\frac{\alpha+\varepsilon}{\Sigma_t}}}{k!}
\Big(1+\frac{1}{k^\frac{\alpha+\varepsilon}{\Sigma_t}}
\Big)\cdots\Big(1+\frac{k-1}{k^\frac{\alpha+\varepsilon}{\Sigma_t}}\Big)\\
&\leq2^{k}\cdot e^{k\frac{\alpha+\varepsilon}{\Sigma_t}}\cdot(k!)^{\frac{\alpha+\varepsilon}{\Sigma_t}-1}.
\end{align}
In view of \eqref{lens}, we have
\begin{equation}\label{ik2}
|I_k(a_1, \ldots, a_k)| \leq (C_2)^k \Big(\prod^k_{i=1}a_i\Big)^{-d}\leq(C_2)^k\big((k-m)!\big)^{-d\cdot\frac{\alpha-\varepsilon}{\Sigma_t}}.
\end{equation}
Taking
\[ds\cdot\frac{\alpha-\varepsilon}{\Sigma_t}=\frac{\alpha+\varepsilon}{\Sigma_t}-1+\varepsilon,\]
we conclude from \eqref{b1}, \eqref{ges2} and \eqref{ik2} that
\begin{align*}
\mathcal{H}^{s}(B_{1}(\alpha,\{t_i\}_{0\leq i\leq m}))&\leq\liminf\limits_{n\to\infty}\sum\limits_{k\geq n}\sum\limits_{(a_1, \ldots, a_{k})\in\widetilde{\mathcal{D}_{k}}(\alpha,\{t_i\}_{0\leq i\leq m})}|I_{k}(a_1, \ldots, a_k)|^s\\
&\leq\liminf\limits_{n\to\infty}\sum\limits_{k\geq n}\Big(\sharp\widetilde{\mathcal{D}_{k}}(\alpha,\{t_i\}_{0\leq i\leq m})\cdot|I_{k}(a_1, \ldots, a_k)|^s\Big)\\
&\leq\liminf\limits_{n\to\infty}\sum\limits_{k\geq n}\frac{(2C^s_2)^{k}\cdot e^{k\frac{\alpha+\varepsilon}{\Sigma_t}}\cdot(k^m)^{\frac{\alpha+\varepsilon}{\Sigma_t}-1}}
{\big((k-m)!\big)^\varepsilon}=0,
\end{align*}
which implies that
\[\dim_{\rm H} B_{1}(\alpha,\{t_i\}_{0\leq i\leq m})
\leq s=\frac{\alpha+\varepsilon+(-1+\varepsilon)\Sigma_t}{d(\alpha-\varepsilon)}.\]
Letting $\varepsilon\rightarrow0^+$, we deduce from \eqref{wsb} that
\[\dim_{\rm H}E(\Lambda,\alpha,\{t_i\}_{0\leq i\leq m})\leq\dim_{\rm H} B_{1}(\alpha,\{t_i\}_{0\leq i\leq m})\leq\frac{\alpha-\Sigma_t}{d\alpha}.\]

\begin{rem}\label{55}
Let $\Sigma_t\leq\alpha_1\leq\alpha_2<\infty$ and let
\[
E(\Lambda,\alpha_1,\alpha_2,\{t_i\}_{0\leq i\leq m})
=\left\{x\in\Lambda: \alpha_1\leq\liminf\limits_{n\to\infty}\frac{\log\left(a_n^{t_0}(x)\cdots a_{n+m}^{t_m}(x)\right)}{\log n}\leq\alpha_2\right\}.
\]
With the same method for estimating the upper bound of $\dim_{\rm H}E(\Lambda,\alpha,\{t_i\}_{0\leq i\leq m})$, we can change the parameter $\alpha$ to $\alpha_1,\alpha_2$ in the corresponding places in \eqref{b1}, \eqref{ges2} and \eqref{ik2}, and then show that
\begin{equation}\label{pha}
\dim_{\rm H}E(\Lambda,\alpha_1,\alpha_2,\{t_i\}_{0\leq i\leq m})\leq\frac{\alpha_2-\Sigma_t}{d\alpha_1},
\end{equation}
\end{rem}

\medskip

\textbf{Lower bound:}  In what follows, we always assume that $\alpha>\Sigma_t$ from the upper bound estimate. Let
\[\{s_n\}_{n\geq1}=\{(2n+1)n^{\frac{\alpha}{\Sigma_t}-1}\}_{n\geq1}\ \text{and}\ \{r_n\}_{n\geq1}=\{ n^{\frac{\alpha}{\Sigma_t}-1}\}_{n\geq1}.\]
Then we claim that
\begin{equation}\label{zj2}
B(\{s_n\},\{r_n\},1)\subseteq E(\Lambda,\alpha,\{t_i\}_{0\leq i\leq m}).
\end{equation}
By \eqref{zj2} and Lemma \ref{sbk}, we have
\begin{align*}
\dim_{\rm H}E(\Lambda,\alpha,\{t_i\}_{0\leq i\leq m})&\geq\dim_{\rm H}B(\{s_n\},\{r_n\},1)\\
&=\liminf\limits_{n\to\infty}\frac{\sum_{1\leq i\leq n}\log r_i}{d\sum_{1\leq i\leq n+1}\log s_i-\log r_{n+1}}\\
&=\frac{\alpha-\Sigma_t}{d\alpha}.
\end{align*}

\medskip

\subsubsection{\textbf{Case} $\alpha=\infty$}
In this case, we easily obtain that
\[
\dim_{\rm H}E(\Lambda,\infty,\{t_i\}_{0\leq i\leq m})\leq\dim_{\rm H}\Lambda=\frac{1}{d}.
\]
For the lower bound $\dim_{\rm H}E(\Lambda,\infty,\{t_i\}_{0\leq i\leq m})$, we choose
\[
\{s_n\}_{n\geq1}=\{(2n+1)e^n\}_{n\geq1}\ \text{and}\ \{r_n\}_{n\geq1}=\{e^n\}_{n\geq1}.
\]
Then it is clear that
\begin{equation}\label{wq}
B(\{s_n\},\{r_n\},1)\subseteq E(\Lambda,\infty,\{t_i\}_{0\leq i\leq m}).
\end{equation}
By Lemma \ref{sbk} again,
\[
\dim_{\rm H}E(\Lambda,\infty,\{t_i\}_{0\leq i\leq m})\geq\dim_{\rm H}B(\{s_n\},\{r_n\},1)
=\frac{1}{d}.
\]
\medskip

\subsection{Proof of Theorem \ref{supp}}
For any $0\leq\alpha\leq\infty$, it is clear that
\[F(\Lambda,\alpha,\{t_i\}_{0\leq i\leq m})\subseteq E(\Lambda,\alpha,\{t_i\}_{0\leq i\leq m}).\]
Then from the results in Theorem \ref{jqdg}, we obtain the desired upper bound of $\dim_{\rm H}F(\Lambda,\alpha,\{t_i\}_{0\leq i\leq m})$. For the lower bound of $\dim_{\rm H}F(\Lambda,\alpha,\{t_i\}_{0\leq i\leq m})$ and $\dim_{\rm H}G(\Lambda,\alpha,\{t_i\}_{0\leq i\leq m})$, we remark that the sets $B(\{s_n\},\{r_n\},1)$, constructed in \eqref{zj2} and \eqref{wq} for $\Sigma_{t}<\alpha<\infty$ and $\alpha=\infty$ respectively, also satisfy
\[B(\{s_n\},\{r_n\},1)\subseteq F(\Lambda,\alpha,\{t_i\}_{0\leq i\leq m})\subseteq G(\Lambda,\alpha,\{t_i\}_{0\leq i\leq m}).\]
The following is to estimate the upper bound of $\dim_{\rm H}G(\Lambda,\alpha,\{t_i\}_{0\leq i\leq m})$. We divide the proof into two cases. For the case $\alpha=\infty$,
\[\dim_{\rm H}G(\Lambda,\infty,\{t_i\}_{0\leq i\leq m})\leq\dim_{\rm H}\Lambda=\frac{1}{d}.\]
For the case $0\leq\alpha<\infty$, we remark that
\begin{equation}\label{zzj}
\ G(\Lambda,\alpha,\{t_i\}_{0\leq i\leq m})
\subseteq\left\{x\in\Lambda: \liminf\limits_{n\to\infty}\frac{\log\left(a_n^{t_0}(x)\cdots a_{n+m}^{t_m}(x)\right)}{\log n}\leq\alpha\right\}.
\end{equation}
Notice that the set on the right-hand side of \eqref{zzj} can be represented as
\[\left\{x\in\Lambda: \liminf\limits_{n\to\infty}\frac{\log\left(a_n^{t_0}(x)\cdots a_{n+m}^{t_m}(x)\right)}{\log n}<\Sigma_t\right\}\bigcup E(\Lambda,\Sigma_t,\alpha,\{t_i\}_{0\leq i\leq m}),\]
where the set $E(\Lambda,\Sigma_t,\alpha,\{t_i\}_{0\leq i\leq m})$ is defined as in Remark \ref{55}. It is worth pointing out that
\begin{align*}
&\ \left\{x\in\Lambda: \liminf\limits_{n\to\infty}\frac{\log\left(a_n^{t_0}(x)\cdots a_{n+m}^{t_m}(x)\right)}{\log n}<\Sigma_t\right\}\\
&=\bigcup_{\ell\geq1}\left\{x\in\Lambda: \liminf\limits_{n\to\infty}\frac{\log\left(a_n^{t_0}(x)\cdots a_{n+m}^{t_m}(x)\right)}{\log n}\leq\Sigma_t-\frac{1}{\ell}\right\}.
\end{align*}
Then by Remark \ref{45} and \eqref{zzj}, it is sufficient to show the upper bound of $\dim_{\rm H}E(\Lambda,\Sigma_t,\alpha,\{t_i\}_{0\leq i\leq m})$. By \eqref{pha}, we assume that
$\Sigma_t<\alpha<\infty$.
\begin{lem}\label{Key}
For any $\Sigma_t<\alpha<\infty$, we have
\begin{equation*}
\dim_{\rm H}E(\Lambda,\Sigma_t,\alpha,\{t_i\}_{0\leq i\leq m})\leq
\frac{\alpha-\Sigma_t}{d\alpha}.
\end{equation*}
\end{lem}
\begin{proof}
For any positive integer $n$ such that
\begin{equation}\label{n5}
n>\frac{\alpha-\Sigma_t}{\Sigma_t},
\end{equation}
we have
\begin{align*}
&\ E(\Lambda,\Sigma_t,\alpha,\{t_i\}_{0\leq i\leq m})\\
&=\bigcup_{0\leq k\leq n-1}E\Big(\Lambda,\Sigma_t+\frac{k}{n}(\alpha-\Sigma_t),\Sigma_t+\frac{k+1}{n}(\alpha-\Sigma_t),\{t_i\}_{0\leq i\leq m}\Big),
\end{align*}
which, in combination with Remark \ref{55}, implies that
\begin{align}\label{ws4}
\nonumber&\ \dim_{\rm H}E(\Lambda,\Sigma_t,\alpha,\{t_i\}_{0\leq i\leq m})\\
\nonumber&=\max_{0\leq k\leq n-1}\dim_{\rm H}E\Big(\Lambda,\Sigma_t+\frac{k}{n}(\alpha-\Sigma_t),\Sigma_t+\frac{k+1}{n}(\alpha-\Sigma_t),\{t_i\}_{0\leq i\leq m}\Big)\\
&\leq\max_{0\leq k\leq n-1}\Big\{\frac{(k+1)(\alpha-\Sigma_t)}{d((n-k)\Sigma_t+k\alpha)}\Big\}.
\end{align}
Let
\[f(k)=\frac{(k+1)(\alpha-\Sigma_t)}{d((n-k)\Sigma_t+k\alpha)},\ \ 0\leq k\leq n-1.\]
Then by \eqref{n5}, the function $f(k)$ is increasing on the interval $[0,n-1]$. Thus, we conclude from \eqref{ws4} that
\[\dim_{\rm H}E(\Lambda,\Sigma_t,\alpha,\{t_i\}_{0\leq i\leq m})\leq f(n-1)=\frac{n(\alpha-\Sigma_t)}{d\Sigma_t+(n-1)d\alpha}.\]
By letting $n\to\infty$, we obtain the desired upper bound.
\end{proof}

\bigskip

{\bf Acknowledgement:}
The authors are grateful to Professor Lingmin Liao for helpful discussions. The research is supported by the National Natural Science Foundation of China (Nos.\, 12101191, 12201207, 12371072).

\end{document}